\documentclass[a4paper,11pt,reqno]{amsart}

\usepackage{amsmath}
\usepackage{amsfonts}
\usepackage{amssymb}
\usepackage{amsthm}
\usepackage[shortlabels]{enumitem}
\usepackage{graphicx}
\usepackage{color}
\usepackage{dsfont}
\usepackage[latin1]{inputenc}
\usepackage{MnSymbol}
\usepackage{enumitem}
\usepackage{extpfeil}
\usepackage{mathtools}
\usepackage{url}
\usepackage[margin=1.2in]{geometry}
\usepackage{fancyhdr}
\usepackage[all,cmtip]{xy}

\numberwithin{equation}{section}

\title{On the Dehn functions of K\"ahler groups}

\author{Claudio Llosa Isenrich}
\address{Laboratoire de Math\'ematiques d'Orsay, Univ. Paris-Sud, CNRS, Universit\'e Paris-Saclay, 91405 Orsay, France}
\email{claudio.llosa-isenrich@math.u-psud.fr}
\author{Romain Tessera}
\address{Institut de Math\'ematiques de Jussieu-PRG, Universit\'e Paris-Diderot, Case 7012, 75205 Paris Cedex 13, France}
\email{romatessera@gmail.com}

\thanks{The first author was supported by a public grant as part of the FMJH. The second author was supported by the grant ANR-14-CE25-0004 ``GAMME''}

\begin{document}

\newcommand{\AAA}{{\mathds A}}
\newcommand{\CC}{{\mathds C}}
\newcommand{\PP}{{\mathbf P}}
\newcommand{\QQ}{{\mathds Q}}
\newcommand{\RR}{{\mathds R}}
\newcommand{\NN}{{\mathds N}}
\newcommand{\ZZ}{{\mathds Z}}
\newcommand{\del}{{\partial}}
\newcommand{\one}{{\mathds {1}}}
\newcommand{\ord}{{\mathcal {O}}}
\newcommand{\ii}{{\mathds {i}}}
\newcommand{\vol}{{\mathrm {vol}}}
\newcommand{\eps}{{\epsilon}}
\def\Mod{{\rm{Mod}}}
\def\C{{\mathds C}}
\def\D{\rm D}
\def\S{\Sigma}
\def\F{{\mathds F}}
\def\FF{\mathcal F}
\def\aut{{\rm{Aut}}}
\def\inn{{\rm{Inn}}}
\def\out{{\rm{Out}}}
\def\isom{{\rm{Isom}}}
\def\mcg{{\rm{MCG}}}
\def\ker{{\rm{ker}}}
\def\im{{\rm{im}}}
\def\dim{{\rm{dim}}}
\def\G{\Gamma}
\def\a{\alpha}
\def\g{\gamma}
\def\L{\Lambda}
\def\Z{{\mathds{Z}}}
\def\H{{\mathds{H}}}
\def\nn{{\bf N}}
\newcommand{\mm}{{\underline{m}}}

\theoremstyle{plain}
\newtheorem{theorem}{Theorem}[section]
\newtheorem{acknowledgement}[theorem]{Acknowledgement}
\newtheorem{claim}[theorem]{Claim}
\newtheorem{conjecture}[theorem]{Conjecture}
\newtheorem{corollary}[theorem]{Corollary}
\newtheorem{exercise}[theorem]{Exercise}
\newtheorem{lemma}[theorem]{Lemma}
\newtheorem{proposition}[theorem]{Proposition}
\newtheorem{question}{Question}
\newtheorem*{question*}{Question}
\newtheorem{addendum}[theorem]{Addendum}

\theoremstyle{definition}
\newtheorem{remark}[theorem]{Remark}
\newtheorem*{acknowledgements*}{Acknowledgements}
\newtheorem{example}[theorem]{Example}
\newtheorem{definition}[theorem]{Definition}
\newtheorem*{notation*}{Notation}
\newtheorem*{convention*}{Convention}

\renewcommand{\proofname}{Proof}

\begin{abstract}
We address the problem of which functions can arise as Dehn functions of K\"ahler groups. We explain why there are examples of K\"ahler groups with linear, quadratic, and  exponential Dehn function. We then proceed to show that there is an example of a K\"ahler group which has Dehn function bounded below by a cubic function and above by $n^6$. As a consequence we obtain that for a compact K\"ahler manifold having non-positive holomorphic bisectional curvature does not imply having quadratic Dehn function.
\end{abstract}

\maketitle

\section{Introduction}

A \textit{K\"ahler group} is a group which can be realized as fundamental group of a compact K\"ahler manifold. K\"ahler groups form an intriguing class of groups. A fundamental problem in the field is Serre's question of ``which" finitely presented groups are K\"ahler. While on one side there is a variety of constraints on K\"ahler groups, many of them originating in Hodge theory and, more generally, the theory of harmonic maps on K\"ahler manifolds, examples have been constructed that show that the class is far from trivial.  Filling the space between examples and constraints turns out to be a very hard problem. This is at least in part due to the fact that the range of known concrete examples and construction techniques are limited. For general background on K\"ahler groups see \cite{ABCKT-95} (and also \cite{Bur-10, BisMj-17} for more recent results).

Known constructions have shown that K\"ahler groups can present the following  group theoretic properties: they can 
\begin{itemize}
 \item be non-residually finite \cite{Tol-93} (see also \cite{CatKol-92}); 
 \item be nilpotent of class 2 \cite{Cam-95, SomVdV-86};
 \item admit a classifying space with finite $k$-skeleton, but no classifying space with finitely many $k+1$-cells \cite{DimPapSuc-09-II} (see also \cite{BisMjPan-14, Llo-16-II, BriLlo-16}); and
 \item be non-coherent \cite{Kap-13} (also \cite{Py-16, FriVid-17}).
\end{itemize}

On the other side strong constraints on K\"ahler groups exclude many groups from being K\"ahler. One of the simplest constraints is that their first Betti number must be even, meaning that for example free abelian groups of odd rank and free groups can not be K\"ahler. Other constraints include that K\"ahler groups are one-ended \cite{Gro-89}, are virtually nilpotent if they are virtually solvable \cite{Del-08} and have quadratically presented Malcev algebra \cite{DGMS-75}.

A fundamental object of study in asymptotic group theory is the Dehn function of a finitely presented group. This relates to other questions in Group theory, such as solvability of the word problem. In this work we address the following question.
\begin{question}
 Which functions can be realised as Dehn functions of K\"ahler groups?
\label{qnDehnKG}
\end{question}

 A straightforward search combining known examples of K\"ahler groups with classical results on Dehn functions shows that they include groups with linear, quadratic, and exponential Dehn function (see Section \ref{secLinQuadExp}). Since all three functions  occur as Dehn functions of many lattices in semi-simple Lie groups \cite{LeuPit-96, Dru-04, LeuYou-17}, this is may not be too surprising, considering that many of the known examples of K\"ahler groups arise from such lattices (see e.g. \cite{Tol-90}). However, these three functions only cover a small fraction of the functions that can be obtained as Dehn functions of finitely presented groups and they play a very special role among them: hyperbolic groups can be characterised by the property that their Dehn function is linear \cite{Gro-87} and there is no non-hyperbolic group with subquadratic Dehn function \cite{Gro-87,Bow-95,Ols-91,CooDelPap-06}. In contrast, the set of $\alpha\geq 2$ such that there is a group with Dehn function $n^{\alpha}$ is dense in $\left[2,\infty\right)$ \cite{BraBri-00}. 

Hence, it is natural to ask if K\"ahler groups can attain any Dehn functions other than linear, quadratic and exponential, and the main part of our work will be dedicated to proving that the answer to this question is positive. This is achieved by showing that an example of a K\"ahler subgroup of a direct product of three surface groups, which was constructed by Dimca, Papadima and Suciu \cite{DimPapSuc-09-II}, has Dehn function bounded below by a cubic function and above by $n^6$. 

\begin{theorem}
 There is a K\"ahler group $G$ with the following properties:
 \begin{enumerate}
 \item $G$ has Dehn function $\delta_G$ with $n^3\preccurlyeq \delta_G(n) \preccurlyeq n^6$;
 \item $G$ is \textit{not coherent};
 \item $G$ is of type $\mathcal{F}_2$ and not of type $\mathcal{F}_3$.
 \end{enumerate}
 \label{thmMainTheorem}
\end{theorem}

It turns out that the group $G$ of Theorem \ref{thmMainTheorem} can be realised as fundamental group of a smooth projective variety $X$ with Stein universal cover $\widetilde{X}$ \cite{DimPapSuc-09-II}. To obtain Theorem \ref{thmMainTheorem} we generalise work by Dison on the Dehn function of subgroups of direct products of free groups \cite{Dis-09}. While the upper bound is a straightforward consequence of his work, the lower bound requires some new ideas. 

Our key estimate to obtain the lower bound is Proposition \ref{prop:distorsion/Dehn}, which gives a lower bound on the distorsion of certain elements in a fibre product in terms of the Dehn function of the common quotient. As pointed out to us by the referee, this can be interpreted as a quantitative version of a result of Mihailova \cite{Mih} proving that there is a subgroup of a direct product of two free groups with unsolvable generalized word problem. We refer the reader to \cite[Section 4]{Mil} for Mihailova's result and several related unsolvability results for subgroups of direct products of two free groups which can be proved using similar methods. In this context we also want to point out that the finite generation and finite presentability of fibre products in terms of the finiteness properties of the groups involved in their construction has been studied in detail. In particular, there are explicit algorithms that, given some explicit data and conditions on two homomorphisms $\phi_i: G_i \to Q$ ($i=1,2$), provide a finite presentation for their fibre product (see \cite{BriHowMilSho-13} and also \cite{BauBriMilSho-00,Dis-08}).

Recall that a closed Riemannian manifold with non-positive sectional curvature admits a quadratic isoperimetric function.
A natural notion of curvature in the context of K\"ahler manifolds is that of holomorphic bisectional curvature. In contrast with the Riemannian setting, we deduce from Theorem \ref{thmMainTheorem} the following result.

\begin{corollary}
 There is a compact K\"ahler manifold with non-positive holomorphic bisectional curvature which does not admit a quadratic isoperimetric function. 
 \label{corHolCurv}
\end{corollary}

\noindent \textbf{Structure.} In Section \ref{secDehnFct} we summarize results on Dehn functions that we shall need. In Section \ref{secLinQuadExp} we provide examples of K\"ahler groups with linear, quadratic and exponential Dehn function. In Section \ref{secDistFibProd} we give a lower bound on the distortion of certain elements in fibre products of hyperbolic groups. We apply this bound in Section \ref{secLowerBound} to show that there is a K\"ahler group with Dehn function bounded below by a cubic function. In Section \ref{secUpperBound} we give an upper bound on the Dehn function of this group allowing us to prove Theorem \ref{thmMainTheorem}. We finish by listing some open questions arising from our work in Section \ref{secQuestions}.

\noindent \textbf{Acknowledgments.} We would like to thank Pierre Pansu for helpful comments and suggestions and the referee for their valuable remarks.

\section{Dehn functions}
\label{secDehnFct}

Let $G$ be a finitely presented group and let $G\cong \left\langle X\mid R\right\rangle$ be a finite presentation for $G$. A \textit{word $w(X)$} of \textit{length} $l(w(X))=m$ in the alphabet $X$ is an expression of the form $w(X) = x_1^{\pm 1}\dots x_m^{\pm 1}$ with $x_i\in X$ for $1\leq i \leq m$. We call a word $w(X)$ \textit{null-homotopic in $G$} if it represents the trivial element in $G$. Every null-homotopic word is freely equal to a word of the form $\prod _{j=1}^k u_j(X) r_j^{\pm 1} (u_j(X))^{-1}$ for some words $u_i(X)$ and elements $r_i\in R$. The area of a null-homotopic word $w(X)$ is
\[
\mathrm{Area}(w(X))= \mathrm{min} \left\{k \mid w(X) \stackrel{F(X)}{=} \prod _{j=1}^k u_j(X) r_j^{\pm 1} (u_j(X))^{-1}\right\}
\]

For non-decreasing functions $f,g: \NN \to \RR_{\geq 0}$ (or $:\RR_{\geq 0} \to \RR_{\geq 0}$) we say that $f$ is \textit{asymptotically bounded} by $g$ if there are constants $C_1, C_2, C_3 \geq 1$ such that $f(n)\leq C_1 g(C_2n)+C_3$ for all $n\in \NN$ (or $n\in \RR$). We write $f\preccurlyeq g$ if $f$ is asymptotically bounded by $g$. We further say that $f$ is \textit{asymptotically equal (or asymptotically equivalent)} to $g$ and write $f\asymp g$ if $f\preccurlyeq g \preccurlyeq f$.

For an element $g\in G$ we denote by $|g|_G= \mathrm{dist}_{Cay(G,X)}(1,g)$ its distance from the origin in the Cayley graph $Cay(G,X)$ with respect to the generating set $X$ of $G$. 

The \textit{Dehn function} of $G$ is the function
\[	
 \delta_G(n)= \mathrm{max} \left\{ \mathrm{Area}(w(X)) \left\mid w(X) ~\mbox{ null-homotopic with}~ l(w(X))\leq n \right. \right\}.
\]
We say that a function $f: \NN \to \RR_{>0}$ is an \textit{isoperimetric function} for $G$ if $\delta_G(n)\preccurlyeq f(n)$. 

For a subgroup $N\leq G$ of a finitely generated group $G$ and generating sets $X$ of $G$ and $Y$ of $N$, the \textit{distortion} of $N$ in $G$ is the function $\Delta_N^G(n)= \mathrm{max} \left\{|g|_H \mid g \in N, |g|_G\leq n \right\}$.

A priori the definitions of $|\cdot|_{G}$ and $\delta_G$, and $\Delta_N^G$ depend on a choice of a finite presentation for $G$ and a finite generating set for $N\leq G$. However, up to asymptotical equivalence, they are independent of these choices and hence it makes sense to speak of them as properties of a finitely presented group rather than of a finite presentation.

We want to summarize a few important properties of Dehn functions which we will require.

\begin{theorem}
Let $G$ be a finitely presented group and $\delta_G$ be its Dehn function. Then the following hold:
\begin{enumerate}
\item $\delta_G$ is linear if and only if $G$ is hyperbolic;
\item if $\delta_G$ is subquadratic (i.e. $\delta_G(n)\prec n^2$) then it is linear;
\item if $G=\pi_1 X$ for $X$ a closed non-positively curved Riemann manifold or, more generally, if $G$ is CAT(0), then the Dehn function of $G$ is at most quadratic. In particular, if $G$ is not hyperbolic then $\delta_G(n)\asymp n^2$;
\item if $G=G_1 \times G_2$ is a direct product of two infinite groups then $$\delta_G(n)\asymp \mathrm{max}\left\{n^2, \delta_{G_1}(n),\delta_{G_2}(n)\right\}.$$
\end{enumerate}
\label{thmDehnFunctions}
\end{theorem}
\begin{proof}
For (1) see \cite{Gro-87}. For (2) see \cite{Gro-87} and also \cite{Ols-91,Bow-95, CooDelPap-06}. For (3) see \cite[III.$\Gamma$.1.6]{BriHae-99}. For (4) see \cite{Bric-93}.
\end{proof}

\section{Linear, quadratic and exponential Dehn functions}
\label{secLinQuadExp}

Comparing the examples of K\"ahler groups in the literature to results on Dehn functions of hyperbolic groups, non-positively curved groups and lattices, it is not hard to see that they include examples with linear, quadratic and exponential Dehn function. In this section we want to give an overview of known examples of K\"ahler groups for which we could determine their Dehn function.

Many of these examples rely on a result by Toledo.

\begin{proposition}[{\cite{Tol-90}}]
 Let $G$ be a semi-simple Lie group with associated symmetric space $X$. Assume that $X$ is an irreducible Hermitian symmetric space and that $X$ is neither the one- or two-dimensional complex unit ball, nor the Siegel upper half plane of genus $2$. Then every non-uniform lattice $\Gamma \leq G$ is a K\"ahler group.
 \label{propToledo}
\end{proposition}

Note that Toledo also shows that non-uniform lattices in $SU(1,1)$ and $SU(2,1)$ are not K\"ahler, while it is not known if non-uniform lattices in $Sp(4,\RR)$ are K\"ahler \cite{Tol-90}.

\subsection*{Linear Dehn function:}
By Theorem \ref{thmDehnFunctions}(1), classifying K\"ahler groups with linear Dehn function is equivalent to classifying hyperbolic K\"ahler groups. Hyperbolic K\"ahler groups include the fundamental groups $\G_g=\pi_1 S_g$ of closed orientable surfaces $S_g$ of genus $\geq 2$ and cocompact lattices $\Gamma \leq PU(n,1)$, which correspond to compact complex ball quotients $\mathds{B}^n/\Gamma$ (see \cite[Section 2]{Par-09} and also \cite{CarSte-10,Sto-15} for examples of such lattices). As an immediate consequence we obtain that there are K\"ahler groups with linear Dehn function.

\subsection*{Quadratic Dehn function:}
It is easy to obtain K\"ahler groups with quadratic Dehn function by taking direct products of hyperbolic groups and applying Theorem \ref{thmDehnFunctions}(4). In particular, we obtain that $\ZZ^{2n}$ and $\G_{g_1}\times \dots \times \G_{g_r}$ have quadratic Dehn function for $n\geq 1$, $r\geq 2$, and $g_i\geq 2$. Searching a bit further we can also find examples of K\"ahler groups with quadratic Dehn function which do not (virtually) decompose as a direct product. They include:

\begin{itemize}
 \item irreducible non-hyperbolic cocompact lattices in semi-simple Lie groups whose associated symmetric space is Hermitian, since non-compact symmetric spaces are CAT(0) and thus all cocompact lattices have Dehn function bounded above by a quadratic function by Theorem \ref{thmDehnFunctions}(3);
 \item the symplectic groups $Sp(2g,\ZZ)$ which have quadratic Dehn function  for $g\geq 5$ by a result of Cohen \cite{Coh-17}, and are known to be K\"ahler for $g\geq 3$ by Proposition \ref{propToledo};
 \item Heisenberg groups $H_{2k+1}$ for $k\geq 4$, which are K\"ahler if and only if $k\geq 4$  (\cite{Cam-95, SomVdV-86} and \cite{CarTol-95}) and have quadratic Dehn function for $k\geq 2$ \cite{All-98};
 \item the compactifications of non-arithmetic lattices constructed in Py's work \cite{Py-16}, since they are non-positively curved and non-hyperbolic (the fundamental group of the cusps maps to $\ZZ^2$-subgroups in the compactification).
\end{itemize} 

A further interesting class of K\"ahler groups with quadratic Dehn function are Hilbert modular groups defined by totally real number fields $K/Q$ of degree $n=\left[K:\QQ\right]\geq 3$. More precisely, let $\phi: K \to \RR^n$ be the homomorphism defined by the $n$ non-trivial embeddings of $K$ in $\RR$. Then $\phi$ defines an embedding $PSL(2,K)\to (PSL(2,\RR))^n$ and the image of a cocompact lattice in $PSL(2,K)$ defines a non-uniform lattice in $(PSL(2,\RR))^n$. Such a lattice is called a Hilbert modular group. It is K\"ahler by \cite{Tol-90}. Hilbert modular groups contain $\ZZ^2$ subgroups and are thus not hyperbolic. In particular, their Dehn function can not be linear. On the other hand they are $\QQ$-rank one lattices and thus by work of Drutu \cite{Dru-04} have Dehn function bounded above by $n^{2+\epsilon}$ for every $\epsilon > 0$. In fact Drutu's work shows that the Dehn function of Hilbert modular groups is quadratic (see also \cite{You-14}).

Note that very recently Leuzinger and Young showed that all irreducible non-uniform lattices in a connected center-free semisimple Lie group of $\RR$-rank $\geq 3$ without compact factors have quadratic Dehn function \cite{LeuYou-17}, confirming a Conjecture of Gromov for this very general class of lattices. Hence, any K\"ahler groups of this form also have quadratic Dehn function.

\subsection*{Exponential Dehn function:}

By work of Leuzinger and Pittet \cite{LeuPit-96}, all irreducible non-uniform lattices in semi-simple Lie groups of $\RR$-rank 2 have exponential Dehn function. Hence, to see that there are K\"ahler groups with exponential Dehn function it suffices to find an example of a K\"ahler group which can be realised by such a lattice. A class of such examples is given by non-uniform lattices in $SU(2,n)$ for $n\geq 2$, since this is a semi-simple irreducible Lie group of real rank $2$ (e.g. \cite{Hel-78}) and it follows from Proposition \ref{propToledo} that non-uniform lattices in $SU(2,n)$ are K\"ahler.

The existence of non-uniform lattices in $SU(2,n)$ is well-known. An example is the group $SU(2,n,\ZZ[i])$, which is a non-uniform lattice by Godement's compactness criterion \cite{BorHar-62, MosTam-62}, since it contains the unipotent element
\[
I_{n+2}+\left(\begin{array}{ccccccc}0 & 0 			& 0 		  & 0  & 0 & \cdots & 0          \\
			0 & \frac{1}{2} & \frac{1}{2} & -1 & \vdots &&\vdots  \\
			0 & -\frac{1}{2}& -\frac{1}{2} & 1  & &&\\
			0 & -1          & -1		  & 0 &0&&\\
			0 & \cdots 		& 			  & 0 &0 &\cdots & 0\\
			\vdots & &&&&& \vdots\\
			0 &&&\cdots &&& 0 \\
		\end{array}\right),
\]
where $I_{n+2}$ is the identity matrix in $GL(n+2,\CC)$ (for further details see \cite[Section 3]{McR-11}).

\subsection*{Other Dehn functions:}
This leaves us with the question if there are any K\"ahler groups whose Dehn function does not fall into one of these three categories. While it is possible that such groups can be found among lattices in semi-simple Lie groups whose associated symmetric space is Hermitian, we are not aware of any known K\"ahler example. Indeed, most of the lattices for which the Dehn function is known seem to fall in one of the previous three categories, which might not be too surprising in the light of the results of Leuzinger and Pittet \cite{LeuPit-96}, and Leuzinger and Young \cite{LeuYou-17}. One notable exception is the 3-Heisenberg group which has cubic Dehn function \cite[Chapter 8]{CEDHLPT-92}, \cite{Ger-92}. However, the latter is not K\"ahler.

\vspace{.2cm}
The remaining sections will be dedicated to proving that K\"ahler groups can admit a Dehn function which is not linear, quadratic or exponential, by proving upper and lower bounds on the Dehn function of a concrete example of a K\"ahler subdirect product of three surface groups constructed by Dimca, Papadima and Suciu \cite{DimPapSuc-09-II}.

\section{Distortion of fibre products of hyperbolic groups}
\label{secDistFibProd}

For short exact sequences of groups
\[
 1\rightarrow N_1 \rightarrow G_1 \stackrel{\phi_1}{\rightarrow} Q \rightarrow 1,
\]
\[
1\rightarrow N_2 \rightarrow G_2 \stackrel{\phi_2}{\rightarrow} Q \rightarrow 1,
\]
their \textit{(asymmetric) fibre product} $P\leq G_1 \times G_2$ is the group $P=\left\{ (g_1,g_2)\mid \phi_1(g_1)=\phi_2(g_2)\right\}$. 

Recall the well-known 0-1-2-Lemma.
\begin{lemma}[0-1-2 Lemma]
Consider the short exact sequences defined above. If $Q$ is finitely presented and $G_1$ and $G_2$ are finitely generated, then the fibre product $P\leq G_1 \times G_2$ is finitely generated.
\end{lemma}
The proof is straight-forward by constructing an explicit finite generating set. We will explain its construction in the special case of a symmetric fibre product, that is, $G_1=G_2$ and $\phi_1 = \phi_2$, since we will require it later. The general case is very similar.

Let $\mathcal{X}$ be a finite generating set for $G$. Denote by $\mathcal{X}_i$ the corresponding generating set of $G_i$ and by $\mathcal{X}_{\Delta}$ the one of the diagonal embedding $\mathrm{Diag}(G)\hookrightarrow G\times G=G_1 \times G_2$. 

Since $G$ is a quotient of the free group $F_{\mathcal{X}}$, the same is true for $Q$. For notational purposes we denote by $\mathcal{X}_Q$ the corresponding finite generating set of $Q$. Since $Q$ is finitely presented we can find a finite set of relations $\mathcal{R}_Q$ among the elements of $\mathcal{X}_Q$, such that $Q\cong \left\langle \mathcal{X}_Q\mid \mathcal{R}_Q\right\rangle$. Denote by $\mathcal{R}_1$ the lift of $\mathcal{R}_Q$ to words in $\mathcal{X}_1$ with respect to the canonical identification $\mathcal{X}_1\cong \mathcal{X}_Q$ provided by viewing $G_1$ and $Q$ as consecutive quotients of $F_{\mathcal{X}}=F_{\mathcal{X}_Q}=F_{\mathcal{X}_1}$. Since $N_1 =\ker ~  \phi_1$, it follows that $N_1 = \left\langle\left\langle \mathcal{R}_1\right\rangle\right\rangle$ is finitely generated as normal subgroup of $G_1$. It is now easy to see that a finite generating set of $P$ is given by $\mathcal{X}_{\Delta} \cup \mathcal{R}_1.$

This generating set allows us to give a lower bound for the distortion of $P$ in $G\times G$ in terms of the Dehn function $\delta_Q$ of $Q$ with respect to the presentation $Q\cong \left\langle \mathcal{X}_Q\mid \mathcal{R}_Q\right\rangle$.

\begin{proposition}\label{prop:distorsion/Dehn}
 Let $G=\left\langle \mathcal{X}\mid \mathcal{S}\right\rangle$ be a finitely presented hyperbolic group, let $Q$ be a finitely presented group and let $\phi: G\rightarrow Q$ be an epimorphism. Let $Q=\left\langle \mathcal{X}_Q \mid \mathcal{R}_Q\right\rangle$ be a finite presentation for $Q$ with $\mathcal{S}\subset \mathcal{R}$. Let $P\leq G\times G$ be the symmetric fibre product of $\phi$ and let $h_n=(g_{1,n},1)\in P$ be a sequence of elements of the intersection $P\cap \left(G_1\times \left\{1\right\}\right)$ with $|g_{1,n}|_G \asymp n$. Assume that each $g_{1,n}$ admits a representative word $v_n(X)$ with $|v_n(\mathcal{X})|_{Free(\mathcal{X})}\asymp n$ such that the null-homotopic word $v_n(\mathcal{X}_Q)$ in $Q$ satisfies $\mathrm{Area}(v_n(\mathcal{X}_Q))\asymp \delta_Q(n)$. Then $|h_n|_P\succcurlyeq \delta_Q(n)$.
 \label{propDistFibProd}
\end{proposition}

\begin{proof}

If $Q$ is hyperbolic then the conclusion is trivially true, since we have the asymptotic inequalities $|h_n|_P \geq |g_{1,n}|_G\asymp n \succcurlyeq \delta_Q(n)$, so assume that $Q$ is not hyperbolic.

 Let $G\cong \left\langle \mathcal{X}\mid \mathcal{S}\right\rangle$ be a finite presentation for $G$. Choose a finite presentation $Q=\left\langle \mathcal{X}_Q\mid \mathcal{R}_Q\right\rangle$ as above. We denote by $\mathcal{S}_1\subset \mathcal{R}_1$ the subset corresponding to the subset $\mathcal{S}\subset \mathcal{R}$.
 
 Since $h_n\in P\cap \left(G_1\times \left\{1\right\}\right)$, the image $\phi(h_n)=\phi_1(g_{1,n})\in Q$ represents the trivial word.  There is a word $\omega_n(\mathcal{X}_{\Delta}, \mathcal{R}_1)$ with $h_n=\omega_n(\mathcal{X}_{\Delta}, \mathcal{R}_1)$ in $P$. Since $\left[G_1\times\left\{1\right\},\left\{1\right\}\times G_2\right]=\left\{1\right\}$, we obtain $\omega_n(\mathcal{X}_{\Delta},\mathcal{R}_1)= \omega_n(\mathcal{X}_1,\mathcal{R}_1)\cdot\omega_n(\mathcal{X}_2,1)$. Hence, the word $\omega_n(\mathcal{X}_2,1)$ represents the trivial element in $G_2$.
 
It follows that the word $\omega_n(\mathcal{X}_1,1)$, obtained from $\omega_n(\mathcal{X}_1,\mathcal{R}_1)$ by deleting all occurrences of elements of $\mathcal{R}_1$, represents the trivial word in $G_1$. Hence, $\omega_n(\mathcal{X}_1,1)$ is freely equal to a product of finitely many conjugates of elements of $\mathcal{S}_1$, whose number will be denoted by $k_{2,n}\geq 0$. Since hyperbolic groups have linear Dehn function, we have that for a minimal choice of $k_{2,n}$,
 \begin{equation}
 \label{eqnLength}
 l(\omega_n(\mathcal{X}_1,1))\succcurlyeq \delta_{G_1}(|\omega_n(\mathcal{X}_1,1)|_{\mathrm{Free}(\mathcal{X}_1)})\geq k_{2,n}
 \end{equation}
and thus
\[
l(\omega_n(\mathcal{X}_1,\mathcal{R}_1))= k_{1,n} + l(\omega_n(\mathcal{X}_1,1))\succcurlyeq k_{1,n}+k_{2,n},
\]
where $k_{1,n}\geq 0$ denotes the number of occurences of elements of $\mathcal{R}_1$ in $\omega_n(\mathcal{X}_1,\mathcal{R}_1)$.

However, it follows from the previous paragraph that $\omega_n(\mathcal{X}_1,\mathcal{R}_1)$ is freely equal to a product of $k_{1,n} + k_{2,n}$ conjugates of elements of $\mathcal{R}_1$. Hence, the area of $\omega_n(\mathcal{X}_1,\mathcal{R}_1)$ in $Q\cong \left \langle \mathcal{X}_1\mid \mathcal{R}_1\right\rangle$ provides us with a lower bound on $k_{1,n}+k_{2,n}$ and \eqref{eqnLength} yields
\begin{equation}
\label{eqnArea}
 l(\omega_n(\mathcal{X}_1,\mathcal{R}_1))\succcurlyeq k_{1,n}+k_{2,n}\geq \mathrm{Area}_Q(\omega_n(X_1,\mathcal{R}_1)). 
\end{equation}

The word $\omega_n(\mathcal{X}_1,\mathcal{R}_1)\cdot (v_n(X_1))^{-1}$ is null-homotopic in $G_1$. Thus, it can be written as product of $k_{3,n}\geq 0$ conjugates of elements of $\mathcal{S}_1$, where we choose $k_{3,n}$ to be minimal with this property. Since hyperbolic groups have linear Dehn function, we deduce
\begin{equation}
k_{3,n} \preccurlyeq |\omega_n(\mathcal{X}_1,\mathcal{R}_1)\cdot (v_n(\mathcal{X}_1))^{-1}|_{Free(\mathcal{X}_1)} \preccurlyeq  |\omega_n(\mathcal{X}_1,\mathcal{R}_1)|_{Free(\mathcal{X}_1)} + n.
\label{eqnk3}
\end{equation}

Using that $\delta_Q(n) \asymp \mathrm{Area}_Q(v_n(X))$, we further obtain 
\begin{equation}
 \mathrm{Area}_Q(\omega_n(\mathcal{X}_1,\mathcal{R}_1))+k_{3,n} \geq \mathrm{Area}_Q(v_n(X_1))  \asymp \delta_Q(n)
 \label{eqnk32}
\end{equation}
Combining inequalities \eqref{eqnArea}, \eqref{eqnk3} and \eqref{eqnk32}, this yields
\begin{align*}
\delta_Q(n) & \underset{\mbox{\eqref{eqnk32}}}{\preccurlyeq} \mathrm{Area}_Q(\omega_n(\mathcal{X}_1,\mathcal{R}_1))+k_{3,n}\\
 & \underset{\mbox{\eqref{eqnk3}}}{\preccurlyeq} \mathrm{Area}_Q(\omega_n(\mathcal{X}_1,\mathcal{R}_1)) + n + |\omega_n(\mathcal{X}_1,\mathcal{R}_1)|_{Free(X_1)}\\ 
 &\underset{\mbox{\eqref{eqnArea}}}{\preccurlyeq} 2\cdot l(\omega_n(\mathcal{X}_1,\mathcal{R}_1)) +n.
\end{align*}
Since $Q$ is non-hyperbolic, we obtain that $\frac{\delta_Q(n)}{n} \to \infty$ as $n\to \infty$, and thus $l(\omega_n(\mathcal{X}_1,\mathcal{R}_1))\succcurlyeq \delta_Q(n)$. It follows that $|h_n|_P \succcurlyeq \delta_Q(n)$.
\end{proof}

A special type of fibre products that we are interested in are the \textit{coabelian subgroups} of a direct product of groups, that is, subgroups $H\leq G_1\times \dots \times G_r$ with $H=\ker ~ ~ \theta$ for some epimorphism $\theta : G_1\times \dots \times G_r \rightarrow \ZZ^N$. 

For a group $G$ consider an epimorphism $\phi: G\rightarrow Q=\ZZ^N$ and denote by $\phi_1:G_1\rightarrow Q$, $\phi_2:G_2 \rightarrow Q$ two copies of this epimorphism. Then the coabelian subgroup $K=\ker ~  (\phi_1+\phi_2) \leq G_1\times G_2$ is the fibre product of the short exact sequences
\[
1\rightarrow \ker ~  \phi \rightarrow G \stackrel{\phi}{\rightarrow} Q\rightarrow 1,
\]
\[
1\rightarrow \ker ~  \phi \rightarrow G \stackrel{-\phi}{\rightarrow} Q\rightarrow 1.
\]

We obtain the following consequence of Proposition \ref{propDistFibProd}
\begin{corollary}
\label{corDistFibProd}
  Let $G$ be a finitely presented hyperbolic group, let $Q=\ZZ^N$ for some $N\geq 0$ and let $\phi: G\rightarrow Q$ be an epimorphism. Let $K$ be as defined in the previous paragraph and assume that there is an automorphism $\nu:G\rightarrow G$ such that $(\phi\circ \nu)(g)=-\phi(g)$ for all $g\in G$. 
  
  Let $h_n=(g_{1,n},1)\in K$ be a sequence of element of the intersection $K\cap \left(G_1\times \left\{1\right\}\right)$, such that $|g_{1,n}|_G=n$ and $g_{1,n}$ has the same properties as the element $g_{1,n}$ in Proposition \ref{propDistFibProd}. Then $|h_n|_K\succcurlyeq \delta_Q(n)$.
\end{corollary}
\begin{proof}
 It is immediate from the existence of the automorphism $\nu$ that $K$ is isomorphic to the symmetric fibre product $P$ of the short exact sequence defined by $\phi: G\rightarrow Q$. The automorphism is induced by the automorphism $(\mathrm{id}_G,\nu): G_1\times G_2 \rightarrow G_1 \times G_2$ of the product. The element $h_n$ is invariant under this automorphism. Hence, Proposition \ref{propDistFibProd} implies that $|h_n|_K\asymp |h_n|_P \succcurlyeq \delta_Q(n)$.
\end{proof}

\section{Subgroups of direct products of surface groups}

\label{secLowerBound}

Throughout the remainder of the paper we will fix the convention that $\left[a,b\right]=aba^{-1}b^{-1}$. Now consider a surface group $\G_2=\pi_1 S_2$ of genus $2$ and fix a presentation
\[
\G_2 = \left\langle a_1,b_1,a_2,b_2 \mid \left[a_1,b_1\right]\left[a_2,b_2\right]\right\rangle.
\]
Define an epimorphism
\[
\begin{array}{rcl}
\phi: \G_2 &\to& \ZZ^2 = \left\langle a,b \mid \left[a,b\right]\right\rangle\\
a_i&\mapsto &a\\
b_i&\mapsto &b
\end{array}
\]
and take $r\geq 2$ copies $\phi_i : \G^{(i)}_2\to \ZZ^2$, $1\leq i \leq r$, where we denote the generators of $\G_2^{(i)}$ by $a_{j}^{(i)},~ b_{j}^{(i)}$, $j=1,2$.

As in Section \ref{secDistFibProd}, we define coabelian subgroups $K_r=\ker ~  \theta_r$ for 
\[
\theta_r = \sum_{i=1}^r \phi_i:\G^{(1)}_2\times \dots \times \G^{(r)}_2\rightarrow \ZZ^2.
\]
It is not hard to see that the groups $K_r$ for $r\geq 3$ are in fact explicit realizations of the K\"ahler subgroups of direct products of $r$ surface groups constructed by Dimca, Papadima and Suciu \cite{DimPapSuc-09-II} by taking a 2-fold branched cover of genus 2 of an elliptic curve (see \cite{Llo-16} for more details and an explicit construction of finite presentations for the $K_r$).

The rest of this section will be concerned with proving the following result:
\begin{theorem}
\label{thmMainResult}
 The Dehn function of the K\"ahler group $K_3$ satisfies $\delta_{K_3}(n)\succcurlyeq n^3$.
\end{theorem}

It is straight-forward to check that
\begin{equation}
\label{eqnNuSurf}
\begin{array}{rcl}
\nu : \G_2 &\rightarrow & \G_2\\
a_i&\mapsto & (a_ib_i)a_i^{-1} (a_ib_i)^{-1}\\
b_i&\mapsto & (a_ib_i)b_i^{-1} (a_ib_i)^{-1}\\
\end{array}
\end{equation}
defines an automorphism of $\G_2$ which satisfies $\phi\circ \nu =-\phi$.

\begin{lemma}
Let $Q= \ZZ^2 = \left\langle a,b \mid \left[a,b\right]\right\rangle$ and let $h_m=\left(\left[(a_1^{(1)})^m,(b_2^{(1)})^m\right],1\right)\in K_2$, for $m\geq 1$. Then $|h_m|_{K_2}\succcurlyeq \delta_Q(m) (\asymp m^2)$.
\label{lemSurfnsquared}
\end{lemma}
\begin{proof}
 Since $\phi_1\left(\left[(a_1^{(1)})^m,(b_2^{(1)})^m\right] \right)=1$, we have $h_m\in K_2 \cap \left(\G_2^{(1)}\times \left\{1 \right\} \right)$. 
 
 The group $\G_2^{(1)}$ retracts onto the free subgroup $F_2=\left\langle a_1^{(1)}, b_2^{(1)} \right\rangle\leq \G_2^{(1)}$ via the homomorphism defined by
 \[
 \begin{array}{rcl}
  \G_2^{(1)}&\to & F_2\\
  a_1^ {(1)}&\mapsto& a_1^{(1)}\\
  b_2^{(1)}&\mapsto & b_2^{(1)}\\
  a_2^{(1)},b_1^{(1)}&\mapsto &1.
 \end{array}
 \]
 Hence, $F_2\leq \G_2^{(1)}$ is an undistorted subgroup and in particular we have $|g_m|_{\G_2^{(1)}}\asymp |g_m|_{F_2}=4m$, where the last inequality is realised by the sequence of words {\small{ $v_m(\mathcal{X})=\left[(a_1^{(1)})^m,(b_2^{(1)})^m\right]$}} with $|v_m(\mathcal{X})|_{\mathrm{Free}(\mathcal{X})}= 4m$, for $\mathcal{X}=\left\{a_1^{(1)},b_1^{(1)},a_2^{(1)},b_2^{(1)}\right\}$. It is well-known that $\delta_Q(n)\asymp n^2$ and that the sequence of words $w_m(\left\{a,b\right\})=\left[a^m,b^m\right]$ satisfies $\mathrm{Area}_Q(w_m(\left\{a,b\right\}))\asymp \delta_Q(m)$.
 
Replace the presentation for $Q$ with the presentation 
\[
 Q \cong \left\langle a_1^Q,b_1^Q,a_2^Q,b_2^Q\mid \left[a_1^Q,b_1^Q\right], a_1^Q(a_2^Q)^{-1}, b_1^Q(b_2^Q)^{-1}, \left[a_1^Q,b_1^Q\right]\left[a_2^Q,b_2^Q\right]\right\rangle,
\]
via the identifications $a_i^Q\mapsto a$, $b_i^Q\mapsto b$. Under this change of presentation the word $v_m(\mathcal{X}_Q)$ gets identified with the word $w_m(\left\{a,b\right\})$ under this identification. Thus, $\mathrm{Area}_Q(v_m(\mathcal{X}_Q))\asymp \delta_Q(m)\asymp m^2$.
 
It follows that the homomorphism $\phi: G\rightarrow Q$ as defined at the beginning of this section, $\nu$ as defined in \eqref{eqnNuSurf} and $h_m$ satisfy all conditions of Corollary \ref{corDistFibProd}. We obtain that $|h_m|_{K_2}\succcurlyeq \delta_Q(m) \asymp m^2$.
\end{proof}

\begin{remark}
Note that in fact the same argument works for any sequence of elements $g_m$ which is contained in an undistorted free subgroup $H\leq \G_{2}^{(2)}$, has reduced length asymptotically equivalent to $m$ and can be represented by words $v_m(\mathcal{X})$ satisfying the conditions of Proposition \ref{propDistFibProd}. Thus the conclusions of Lemma \ref{lemSurfnsquared} apply to this more general class of elements in $K_2 \cap \left( \G^{(1)}\times \left\{1\right\}\right)$.
\label{rmkSurfnsquared}
\end{remark}

The rest of the proof of Theorem \ref{thmMainResult} will require a suitable decomposition of $K_3$ to which we can apply the following consequence of \cite[Theorem 6.1]{Dis-09}.

\begin{theorem}
 Let $\Lambda = G_1 \ast_{H} G_2$ be finitely presented with $H$ a proper subgroup of each $G_i$, and $H=\left\langle \mathcal{B}\right\rangle$, $G_1=\left\langle \mathcal{A}_1\right\rangle$ and $G_2=\left \langle \mathcal{A}_2\right\rangle$, where $\mathcal{B}$, $\mathcal{A}_1$ and $\mathcal{A}_2$ are finite generating sets. Let $\mathcal{P}=\left\langle \mathcal{A}_1,\mathcal{A}_2\mid \mathcal{R}\right\rangle$ be a finite presentation for $\Lambda$. Then there is a constant $C=C(\mathcal{P},\mathcal{B})>0$ such that the following holds:
 
 Given elements $h\in H$, $g_1\in G_1 \setminus H$ and $g_2 \in G_2\setminus H$ with $\left[g_1,h\right]=\left[g_2,h\right]=1$, and words $w=w(\mathcal{A}_1)$ representing $h$ and $u_i=u_i(\mathcal{A}_i)$  representing  $g_i$, $i=1,2$, then
 \[
  \mathrm{Area}_{\mathcal{P}}(\left[w,(u_1u_2)^n\right]) \geq C \cdot n \cdot \mathrm{dist}_{\mathrm{Cay}(H,\mathcal{B})}(1,h)
 \]
 
 \label{thmDison}
\end{theorem} 
 
 Note that \cite[Theorem 6.1]{Dis-09} shows that for an explicit presentation of $\Lambda$ of the form $\left\langle \mathcal{A}_1,\mathcal{A}_2,\mathcal{B}\mid \mathcal{R}\right\rangle$ one can choose $C=2$. It is straight-forward to see that Theorem \ref{thmDison} follows from this result, using that $\mathrm{Area}_{\mathcal{P}}$ and $\mathrm{dist}_{\mathrm{Cay}(H,\mathcal{B})}$ are equivalent up to multiplicative constants under change of presentation and generating set.

Projection of the subgroup $K_3\leq \G_2^{(1)}\times \G_2^{(2)}\times \G_2^{(3)}$ onto the third factor induces a short exact sequence
\[
 1 \rightarrow K_2 \rightarrow K_3 \stackrel{p_3}{\rightarrow} \G_2^{(3)}\rightarrow 1
\]
with $K_2=\ker ~ \left(\G_2^{(1)}\times \G_2^{(2)}\rightarrow \ZZ^2\right)$. The homomorphism 
\[
(1,\nu,id_{\G_2}): \G_2^{(3)} \rightarrow K_3
\]
provides a splitting of this sequence. Hence, it follows that $K_3\cong K_2\rtimes \G_2^{(3)}$ is a semidirect product. 

Consider the decomposition $\G_2^{(3)}= \mathrm{Free}(a_1^{(3)},b_1^{(3)})\ast_{\left[a_1^{(3)},b_1^{(3)}\right]=\left[b_2^{(3)},a_2^{(3)}\right]} \mathrm{Free}(a_2^{(3)},b_2^{(3)})$ into an amalgamated product over $\ZZ$. Its lift under $p_3$ provides a decomposition of $K_3$ as amalgamated free product
\[
K_3 \cong \left(K_2 \rtimes \mathrm{Free}(a_1^{(3)},b_1^{(3)})\right)\ast_{H} \left(K_2 \rtimes \mathrm{Free}(a_2^{(3)},b_2^{(3)})\right),
\]
with 
\[
H=p_3^{-1}\left( \left\langle \left[ a_1^{(3)} ,b_1^{(3)} \right] \right\rangle \right) = K_2 \rtimes \left\langle \left[ a_1^{(3)} , b_1^{(3)} \right]\right\rangle \cong K_2 \times \left\langle \left[a_1^{(3)},b_1^{(3)} \right] \right\rangle,
\]
where the last identity follows from the fact that $\left(1,\left[\nu(a_1^{(3)}),\nu(b_1^{(3)})\right],\left[a_1^{(3)},b_1^{(3)} \right]\right)K_2 = \left(1,1,\left[a_1^{(3)},b_1^{(3)} \right]\right)K_2$ as cosets of $K_2\unlhd H$.

To simplify notation we will now write $K_2\times \ZZ$ for $K_2 \times \left\langle \left[a_1^{(3)},b_1^{(3)} \right] \right\rangle$.

\begin{lemma}
\label{lemK2ZUndist}
 Let $h\in (K_2 \times \ZZ)\cap \left(\G_2^{(1)}\times \left\{1\right\}\times \left\{1\right\}\right)$. Then $|h|_{K_2\times \ZZ}\asymp |h|_{K_2}$.
\end{lemma}
\begin{proof}
 Let $\mathcal{Y}$ be a generating set for $K_2$ and let $z=\left(1,1,\left[a_1^{(3)},b_1^{(3)} \right] \right)$. Then $\mathcal{Y}_{+}=\mathcal{Y}\cup \left\{z\right\}$ is a generating set for $K_2\times \ZZ$. It is clear that with respect to this generating set we have $|h|_{K_2\times \ZZ}\leq |h|_{K_2}$.
 
 Conversely, let $\omega(\mathcal{Y},z)$ be a word of minimal length in $K_2\times \ZZ$ representing $h$ in $K_2\times \ZZ$. Then $\omega(\mathcal{Y},z)=\omega(\mathcal{Y},1)\omega(1,z)$ and in particular $\omega(1,z)$ represents the trivial element in $\ZZ$, while $\omega(\mathcal{Y},1)$ is a word in $\mathcal{Y}$ representing $h$.  Hence, we obtain 
 \[
 |h|_{K_2\times \ZZ} = l\left(\omega(\mathcal{Y},z)\right)=l\left(\omega(\mathcal{Y},1)\omega(1,z)\right) \geq l\left(\omega(\mathcal{Y},1)\right)\geq |h|_{K_2}.
 \]
\end{proof}

As a consequence of Lemma \ref{lemK2ZUndist} and Lemma \ref{lemSurfnsquared}, we obtain
\begin{corollary}
For $m\geq 1$, the element $h_m=\left(\left[(a_1^{(1)})^m,(b_2^{(1)})^m\right],1,1\right)\in K_2\times \ZZ$ satisfies $|h_m|_{K_2\times \ZZ}\succcurlyeq m^2$.
\label{corSurfnsquared}
\end{corollary}

\begin{proof}[Proof of Theorem \ref{thmMainResult}]
Let $G_1=K_2 \rtimes \mathrm{Free}(a_1^{(3)},b_1^{(3)})$ and $G_2= K_2 \rtimes \mathrm{Free}(a_2^{(3)},b_2^{(3)})$. Extend the set $\left\{\left(1,\nu(a_1^{(3)}),a_1^{(3)}\right),\left( a_1^{(1)}, 1, (a_1^{(3)})^{-1}\right),\left(b_2^{(1)},(b_2^{(2)})^{-1},1\right)\right\}\subset G_1$ to a finite generating set $\mathcal{A}_1$ of $G_1$ and the set $\left\{ \left(1,\nu(a_2^{(3)}),a_2^{(3)}\right)\right\}\subset G_2$ to a finite generating set $\mathcal{A}_2$ of $G_2$ and choose any finite generating set $\mathcal{B}$ for $H$. Let further $\mathcal{P}=\left\langle \mathcal{A}_1,\mathcal{A}_2\mid \mathcal{R}\right\rangle$ be a finite presentation for $K_3$.

  Choose elements represented by words $$g_1= u_1(\mathcal{A}_1)=\left(1,\nu(a_1^{(3)}),a_1^{(3)}\right)\in \left(K_2 \rtimes \mathrm{Free}(a_1^{(3)},b_1^{(3)})\right) \setminus H$$ and $$g_2=u_2(\mathcal{A}_2)=\left(1,\nu(a_2^{(3)}),a_2^{(3)}\right)\in \left(K_2 \rtimes \mathrm{Free}(a_2^{(3)},b_2^{(3)})\right) \setminus H.$$ Since $h_m\in K_3 \cap \G_2^{(1)}$, we have $\left[g_1,h_m\right]=\left[g_2,h_m\right]=1$. The element $h_m$ is represented by the word 
\[
 w_m(\mathcal{A}_1) = \left[ \left( a_1^{(1)}, 1, (a_1^{(3)})^{-1} \right)^m,\left(b_2^{(1)},(b_2^{(2)})^{-1},1\right)^m \right]
\]
of length $l(w_m)=4m$ with respect to the generating set $\mathcal{A}_1$.

 Thus, by Theorem \ref{thmDison} there is $C=C(\mathcal{P},\mathcal{B})>0$ such that
\[
 \mathrm{Area}_{\mathcal{P}}(\left[w_m,(u_1u_2)^m\right]) \geq C \cdot m \cdot \mathrm{dist}_{\mathrm{Cay}(H,\mathcal{B})}(1,h_m).
\] 
 
 The word $\left[w_m,(u_1u_2)^m\right]$ has length $l(\left[w_m,(u_1u_2)^m\right])= 12m$ in the alphabet $\mathcal{A}_1\cup \mathcal{A}_2$. Hence, it follows from Corollary \ref{corSurfnsquared} that
 \[
 \delta_{K_3}(m)\succcurlyeq \mathrm{Area}_{\mathcal{P}}(\left[w_m,(u_1u_2)^m\right])\succcurlyeq C m |h_m|_{K_2\times \ZZ} \succcurlyeq m^3.
 \]
 This completes the proof of Theorem \ref{thmMainResult}.
\end{proof}

Note that $K_3$ is a K\"ahler group of type $\mathcal{F}_2$, not of type $\mathcal{F}_3$ (see \cite{DimPapSuc-09-II}), which contains the subgroup $K_2$ of type $\mathcal{F}_1$, not of type $\mathcal{F}_2$. In particular, $K_3$ is not coherent. Hence, we can summarize the main properties of $K_3$ as follows.
\begin{theorem}
 The group $K_3$ is a non-coherent K\"ahler group of type $\mathcal{F}_2$, not of type $\mathcal{F}_3$. The Dehn function of $K_3$ is bounded below by $\delta_{K_3}(n)\succcurlyeq n^3$.
 \label{thmLowerBound}
\end{theorem}

As a consequence we obtain:
\begin{proof}[Proof of Corollary \ref{corHolCurv}]
 Dimca, Papadima and Suciu obtain the group $K_3$ as fundamental group of the (compact) smooth generic fibre $H$ of a surjective holomorphic map $f: S_2^{(1)}\times S_2^{(2)}\times S_2^{(3)} \to E$ for $E$ an elliptic curve. In particular, $H \subset S_2^{(1)}\times S_2^{(2)}\times S_2^{(3)}$ is an embedded complex submanifold. The manifold $S_2^{(1)}\times S_2^{(2)}\times S_2^{(3)}$ admits a metric with non-positive holomorphic bisectional curvature, since it is a product of closed Riemann surfaces.   Since holomorphic bisectional curvature is non-increasing when passing to complex submanifolds \cite[Section 4]{GolKob-67}, it follows that $H$ can be endowed with a metric of non-positive holomorphic bisectional curvature. By Theorem \ref{thmLowerBound}, $K_3$ does not admit a quadratic isoperimetric function and the result follows.
\end{proof}

Note that Corollary \ref{corHolCurv} is in contrast to the case of non-positive (Riemannian) sectional curvature, as the latter implies that the fundamental group admits a quadratic isoperimetric function.

\section{An upper bound}
\label{secUpperBound}

For a group $G$ with finite presentation $G=\left\langle X\mid R\right\rangle$, an area-radius pair $\left( \alpha, \rho\right)$ is a pair of functions $\alpha: \NN \to \RR_{>0}$ and $\rho : \NN \to \RR_{>0}$ such that for every null-homotopic word $w(X)$ with $l(w(X)) \leq n$, there is a free equality $w(X) = \prod _{j=1}^k u_i(X) r_i^{\pm 1} (u_i(X))^{-1}$ with $k\leq \alpha(n)$ and $\mathrm{max}\left\{l(u_i(X))\mid 1\leq i \leq k\right\}\leq \rho(n)$. Note that area-radius pairs are independent of the choice of presentation (up to equivalence). For further details on area-radius pairs we refer the reader to \cite[Section 3]{Dis-08}.

An upper bound on the Dehn function of $K_3$ can be obtained as a consequence of the following theorem by Dison:
\begin{theorem}[{\cite[Theorem 11.15]{Dis-08}}]
 For $r\geq 3$, let $G_1, \dots, G_r$ be finitely presented groups with area-radius pairs $\left(\alpha_i,\rho_i\right)$, for $1\leq i \leq r$, and let $A$ be an abelian group of rank $m=\mathrm{dim}\left(A\otimes_{\ZZ} \QQ\right)$. Define $\alpha(n)=\mathrm{max}\left\{ n^2, \alpha_i(n), 1\leq i \leq r\right\}$, $\rho(n)=\mathrm{max}\left\{n, \rho_i(n),1\leq i \leq r\right\}$.
 
 If $\phi: G_1 \times \dots \times G_r\to A$ is a homomorphism such that the restriction of $\phi$ to each factor $G_i$ is surjective, then $\rho^{2m}\alpha$ is an isoperimetric function for $\ker \phi$.
 \label{thmDisAreaRad}
\end{theorem}

\begin{lemma}
 Let $G=\left\langle X \mid R\right\rangle$ be a finitely presented group and let $\delta_G$ be the Dehn function of $G$. Then there is $C>0$ such that $\left(\delta_G(n),C\cdot \delta_G(n) + n\right)$ defines an area-radius pair for $G$.
 \label{lemAreaRadius}
\end{lemma}
\begin{proof}
 It is not hard to see that, given a van Kampen diagram of area $k=\mathrm{Area}(w(X))$ for a word $w(X)$, then $w(X)$ is freely equal to a word of the form 
 \[
  w(X) = \prod _{j=1}^k u_j(X) r_j^{\pm 1} (u_j(X))^{-1}
 \]
with $l(u_i(X))\leq C \cdot k +l(w(X))$, where $C=\mathrm{max} \left\{ l(r(X))\mid r\in R\right\}$. 

Hence, $\left(\delta_G(n),C\delta_G(n) + n\right)$ is an area-radius pair for $G$.
\end{proof}

As an easy consequence we obtain an upper bound on the Dehn function of the group $K_3$.
\begin{corollary}
 The Dehn function of $K_3$ satisfies $\delta_{K_3}(n)\preccurlyeq n^6$.
 \label{corUpperBound}
\end{corollary}
\begin{proof}
 The group $K_3$ is the kernel of the homomorphism $\theta_3: \Gamma_2^{(1)}\times \Gamma_2^{(2)}\times \Gamma_2^{(3)} \to \ZZ^2$ which is surjective on factors. Since surface groups are hyperbolic the groups $\Gamma_2^{(i)}$ have linear Dehn function $\delta_{G_2^{(i)}}(n) = n$. By Lemma \ref{lemAreaRadius} they admit area-radius pairs $\left(\alpha_i,\rho_i\right)$ with $\alpha_i(n)\asymp\rho_i(n)\asymp \delta_{\Gamma_2^{(i)}}(n)\asymp n$. Theorem \ref{thmDisAreaRad} implies that $\delta_{K_3}(n)\preccurlyeq \rho^{4} \cdot \alpha \asymp n^6$.
\end{proof}

Note that Theorem \ref{thmDisAreaRad} was used in a very similar way in \cite[Corollary 12.6]{Dis-08} to show that subdirect products of $r$ limit groups of type $\mathcal{F}_{r-1}$ have polynomial Dehn function and in \cite[Proposition 13.3(2)]{Dis-08} to show that the examples in \cite{Dis-09} have Dehn function bounded above by $n^6$.
\begin{proof}[Proof of Theorem \ref{thmMainTheorem}]
This is now an immediate consequence of Theorem \ref{thmLowerBound} and Corollary \ref{corUpperBound}.
\end{proof}

\section{Questions}
\label{secQuestions}

This work begs intriguing questions and we want to finish by listing some of them. One may first ask whether our result can be extended as follows.
\begin{question}
For every integer $k$, is there $k'\geq k$ and a K\"ahler group $G$ with Dehn function satisfying $n^k\preccurlyeq \delta_G(n) \preccurlyeq n^{k'}$?
\end{question}

Note that the same reasoning as in \cite[Proposition 13.3(3)]{Dis-08} and Section \ref{secUpperBound} can be applied to see that the groups constructed in \cite{DimPapSuc-09-II}, \cite{BisMjPan-14} and \cite{Llo-16-II} must all have Dehn function bounded above by $n^6$. Hence, these examples can not be used to obtain a positive answer to this question. 
However, these arguments do not apply to many of the groups constructed very recently by the first author from maps onto higher-dimensional tori \cite{Llo-17}. Thus, in the light of Theorem \ref{thmMainTheorem}, one could try to search for groups with larger Dehn function among these examples.

Recall that the gap between $1$ and $2$ in Theorem \ref{thmDehnFunctions}(2) is the only gap in the isoperimetric spectrum, that is, for every $\alpha\in \left[2,\infty\right)$ and $\epsilon>0$ there is a group $G$ with Dehn function $n^{\beta}$ and $|\alpha-\beta|<\epsilon$ \cite{BraBri-00}. It is thus natural to ask if the same is true in the class of K\"ahler groups.
\begin{question}
 What is the isoperimetric spectrum of Dehn functions of K\"ahler groups? Does it contain any gaps on $\left[2,\infty\right)$?
\end{question}
We do not know the exact asymptotic of the Dehn function of the group of 
Theorem \ref{thmMainTheorem}, and in particular we do not know whether there exists a single point in the  isoperimetric spectrum of Dehn functions of K\"ahler groups that is different from $1$ or $2$.

We might wonder whether there exist K\"ahler groups with arbitrary large Dehn function. 
For instance
\begin{question}
Do all K\"ahler groups have solvable word problem? 
\end{question}
Recall that the word problem is solvable if and only if the Dehn function is a recursive function.
At this time even the following question is open. 
\begin{question}
Do all K\"ahler groups admit an exponential isoperimetric function?
\end{question}

Finally, an interesting though mysterious class of K\"ahler groups are those which are nilpotent. The only non-trivial known examples are the  higher dimensional Heisenberg groups mentioned in Section \ref{secLinQuadExp} (which have quadratic Dehn function). 
\begin{question}
Do all nilpotent K\"ahler groups  admit a quadratic Dehn function?
\end{question}

More generally, we do not know whether nilpotent groups whose Malcev algebras are quadratically presented admit a quadratic Dehn function, although in the introduction of the first arXiv version of his article \cite{You-06}, Young suggests a strong link between these two conditions (see also \cite[Section 6.2]{Rum-05}).

\bibliography{References}
\bibliographystyle{plain}

\end{document}